\documentclass[oneside,english]{amsart}
\usepackage[T1]{fontenc}
\usepackage[latin9]{inputenc}
\usepackage{geometry}
\usepackage{amstext}
\usepackage{amsthm}
\usepackage{amssymb}

\makeatletter
\numberwithin{equation}{section}
\numberwithin{figure}{section}
\theoremstyle{plain}
\newtheorem{thm}{\protect\theoremname}
  \theoremstyle{plain}
  \newtheorem{conjecture}[thm]{\protect\conjecturename}
  \theoremstyle{plain}
  \newtheorem{lem}[thm]{\protect\lemmaname}
  \theoremstyle{remark}
  \newtheorem{rem}[thm]{\protect\remarkname}

\makeatother

\usepackage{babel}
  \providecommand{\conjecturename}{Conjecture}
  \providecommand{\lemmaname}{Lemma}
  \providecommand{\remarkname}{Remark}
\providecommand{\theoremname}{Theorem}

\begin{document}

\title{Upper Bounds for Sunflower-Free sets}

\author{Eric Naslund, William F. Sawin}
\begin{abstract}
A collection of $k$ sets is said to form a \emph{$k$-sunflower}, or $\Delta$\emph{-system} if the intersection of any two sets
from the collection is the same, and we call a family of sets $\mathcal{F}$
\emph{sunflower-free }if it contains no sunflowers. Following the
recent breakthrough of Ellenberg and Gijswijt and Croot, Lev and Pach
\cite{EllenbergGijswijtCapsets,CrootLevPachZ4} we apply the polynomial
method directly to Erd\H{o}s-Szemer\'{e}di sunflower problem \cite{ErdosSzemeredi1978UpperBoundsForSunflowerFreeSets}
and prove that any sunflower-free family $\mathcal{F}$ of subsets
of $\{1,2,\dots,n\}$ has size at most 
\[
|\mathcal{F}|\leq3n\sum_{k\leq n/3}\binom{n}{k}\leq\left(\frac{3}{2^{2/3}}\right)^{n(1+o(1))}.
\]
We say that a set $A\subset(\mathbb Z/D \mathbb Z)^{n}=\{1,2,\dots,D\}^{n}$
for $D>2$ is sunflower-free if every distinct triple $x,y,z\in A$
there exists a coordinate $i$ where exactly two of $x_{i},y_{i},z_{i}$
are equal. Using a version of the polynomial method with characters
$\chi:\mathbb{Z}/D\mathbb{Z}\rightarrow\mathbb{C}$ instead of polynomials, we show that
any sunflower-free set $A\subset(\mathbb Z/D \mathbb Z)^{n}$ has size 
\[
|A|\leq c_{D}^{n}
\]
where $c_{D}=\frac{3}{2^{2/3}}(D-1)^{2/3}$. This can be seen as making
further progress on a possible approach to proving the Erd\H{o}s-Rado
sunflower conjecture \cite{ErdosRadoTheorem}, which by the work of
Alon, Sphilka and Umans \cite[Theorem 2.6]{AlonSphilkaUmansSunflowerMatrix} is
equivalent to proving that $c_{D}\leq C$ for some constant $C$ independent
of $D$.
\end{abstract}

\date{May 27th 2016}

\email{naslund@princeton.edu, wsawin@math.princeton.edu}

\maketitle

\section{Introduction}

A collection of $k$ sets is said to form a \emph{$k$-sunflower}, or $\Delta$\emph{-system}, if the intersection of any two sets
from the collection is the same. A family of sets $\mathcal{F}$ is
said to be \emph{$k$-sunflower free} if no $k$ members form a $k$-sunflower,
and when $k=3$ we simply say that the collection $\mathcal{F}$ is
\emph{sunflower-free.} It is a longstanding conjecture that sunflower-free
families must be small, and there are two natural situations in which
we may ask this question. The first, and most general case, is when
each set in the family has size $m$. Erd\H{o}s and Rado made the following
conjecture which is now known as the \emph{Sunflower Conjecture.}
\begin{conjecture}
\label{conj:Erdos-Rado-Sunflower-Conjecture}(Erd\H{o}s-Rado Sunflower
Conjecture \cite{ErdosRadoTheorem}) Let $k\geq3$, and suppose that
$\mathcal{F}$ is a $k$-sunflower free family of sets, each of size
$m$. Then 
\[
|\mathcal{F}|\leq C_{k}^{m}
\]
for a constant $C_{k}>0$ depending only on $k$.
\end{conjecture}
In their paper, Erd\H{o}s and Rado \cite{ErdosRadoTheorem} proved that
any $k$-sunflower free family of sets of size $m$ has size at most
$m!(k-1)^{m}$, and the conjectured bound of $C_{k}^{m}$ remains
out of reach for any $k\geq3$. The second setting for upper bounds
for $k$-sunflower-free sets concerns the case where each member of
$\mathcal{F}$ is a subset of the same $n$-element set. There can
be at most $2^{n}$ such subsets, and the Erd\H{o}s-Szemer\'{e}di sunflower
conjecture states that this trivial upper bound can be improved by
an exponential factor.
\begin{conjecture}
\label{conj:Erdos-Szemeredi-Sunflower-Conjecture}(Erd\H{o}s-Szemer\'{e}di
Sunflower Conjecture \cite{ErdosSzemeredi1978UpperBoundsForSunflowerFreeSets})
Let $S$ be a $k$-sunflower free collection of subsets of $\left\{ 1,2,\dots,n\right\} $.
Then 
\[
|S|<c_{k}^{n}
\]
for some constant $c_{k}<2$ depending only on $k$.
\end{conjecture}
In Erd\H{o}s and Szemer\'{e}di's paper \cite{ErdosSzemeredi1978UpperBoundsForSunflowerFreeSets}, they prove that conjecture \ref{conj:Erdos-Szemeredi-Sunflower-Conjecture} follows
from conjecture \ref{conj:Erdos-Rado-Sunflower-Conjecture}   (see also \cite[Theorem 2.3]{AlonSphilkaUmansSunflowerMatrix}), and so
it is a weaker variant of the sunflower problem. Let $F_{k}(n)$ denote
the size of the largest $k$-sunflower-free collection $\mathcal{F}$
of subsets of $\{1,2,\dots,n\}$, and define
\[
\mu_{k}^{S}=\limsup_{n\rightarrow\infty}F_{k}(n)^{1/n}
\]
to be the the \emph{Erd\H{o}s-Szmer\'{e}di-$k$-sunflower-free capacity}.
The trivial bound is $\mu_{k}^{S}\leq2$, and the Erd\H{o}s-Szemer\'{e}di
sunflower conjecture states that $\mu_{k}^{S}<2$ for all $k\geq3$.
In this paper we prove new bounds for the sunflower-free capacity
$\mu_{3}^{S}$. It is a theorem of Alon, Shpilka and Umans \cite[pp. 7]{AlonSphilkaUmansSunflowerMatrix}
that the recent work of Ellenberg and Gijswijt \cite{EllenbergGijswijtCapsets}
on progression-free sets in $\mathbb{F}_{3}^{n}$ implies that $\mu_{3}^{S}<2$,
and before this, the best upper bound for a sunflower-free collection
of $\{1,2,\dots,n\}$ was $2^{n}\exp\left(-c\sqrt{n}\right)$ due
to Erd\H{o}s and Szemer\'{e}di \cite{ErdosSzemeredi1978UpperBoundsForSunflowerFreeSets}. We give a simple proof of a quantitative version of \cite[pp. 7]{AlonSphilkaUmansSunflowerMatrix}, showing that $\mu_3^S \leq \sqrt{1+C}$ where $C$ is the capset capacity. However, using the ideas from the recent breakthrough of Ellenberg and Gijswijt
and Croot, Lev and Pach, \cite{CrootLevPachZ4,EllenbergGijswijtCapsets}
on progressions in $\mathbb{F}_{3}^{n}$, and from Tao's version of the argument \cite{TaosBlogCapsets}, we apply the polynomial
method directly to this problem, and obtain a stronger result:
\begin{thm}
\label{thm:Main-Sunflower-Free-Capacity}Let $\mathcal{F}$ be a sunflower-free
collection of subsets of $\{1,2,\dots,n\}$. Then 
\[
|\mathcal{F}|\leq3(n+1)\sum_{k\leq n/3}\binom{n}{k},
\]
and 
\[
\mu_{3}^{S}\leq\frac{3}{2^{2/3}}=1.889881574\dots
\]

\end{thm}
The best known lower bound for this problem is $\mu_{3}^{S}\geq1.554$
due to the first author \cite{NaslundLowerBounds}, and so there is
still a large gap between upper and lower bounds for the sunflower-free
capacity $\mu_{3}^{S}$.

In section \ref{sec:Sunflower-Free-Sets-in-Z_D} we turn to the sunflower
problem in the set $\{1,2,\dots,D\}$, which we will always think of as $\mathbb Z/D \mathbb Z$ .
Alon, Shpilka and Umans \cite[Definition 2.5]{AlonSphilkaUmansSunflowerMatrix}, defined
a $k$-sunflower in $(\mathbb Z/D \mathbb Z)^{n}$ for $k\leq D$ to be a collection
of $k$ vectors such that in each coordinate they are either all different
or all the same. When $k=3$ and $D=3$ this condition is equivalent
to being a three-term arithmetic progression in $\mathbb{F}_{3}^{n}$. 
\begin{conjecture}
\label{conj:Sunflower-Conjecture-in-Z_D}(Sunflower-Conjecture in
$(\mathbb Z/D \mathbb Z)$) Let $k\leq D$, and let $A\subset(\mathbb Z/D \mathbb Z)^{n}$
be a $k$-sunflower-free set. Then 
\[
|A|\leq b_{k}^{n}
\]
for a constant $b_{k}$ depending only on $k$.
\end{conjecture}
The motivation for this problem comes from \cite[Theorem 2.6]{AlonSphilkaUmansSunflowerMatrix}
where they proved that conjecture \ref{conj:Sunflower-Conjecture-in-Z_D}
is equivalent to the Erd\H{o}s-Rado sunflower conjecture. In particular,
if there exists a constant $C$ independent of $D$ such that any
$3$-sunflower-free set in $(\mathbb Z/D \mathbb Z)^{n}$ has size at most $C^{n}$,
then Conjecture \ref{conj:Erdos-Rado-Sunflower-Conjecture} holds
for $k=3$ with $c_{3}=e\cdot C$. Since a sunflower-free set cannot
contain a $3$-term arithmetic progression, the recent result
of Ellenberg and Gijswijt \cite{EllenbergGijswijtCapsets} implies an
upper bound for sunflower-free sets $A\subset(\mathbb Z/D \mathbb Z)^{n}$ for $D$
prime of the form $|A|\leq c_D^{n}$, where $c_D = D e^{ -I((D-1)/3)}$ for a function $I$ defined in \cite{EllenbergGijswijtCapsets} in terms of a certain optimization problem. It's not too hard to see that
\[0 < \lim_{D \to \infty} \frac{c_D}{D} <1\]
Using the characters $\chi:\mathbb{Z}/D\mathbb{Z}\rightarrow\mathbb{C}$ instead of polynomials,
we prove the following theorem:
\begin{thm}
\label{thm:Main-Z_D-sunflower-free-capacity}Let $D\geq3$, and let $A\subset(\mathbb Z/D \mathbb Z)^{n}$
be a sunflower-free set. Then 
\[
|A|\leq c_{D}^{n}
\]
where $c_{D}=\frac{3}{2^{2/3}}(D-1)^{2/3}.$
\end{thm}
\noindent This can be seen as progress towards the Erd\H{o}s-Rado sunflower
conjecture, and we remark that the now resolved Erd\H{o}s-Szemer\'{e}di conjecture
for $k=3$ is equivalent to proving that $c_{D}<D^{1-\epsilon}$ for some $\epsilon$ and all $D$ sufficiently large \cite[Theorem 2.7]{AlonSphilkaUmansSunflowerMatrix}. 

To prove Theorem \ref{thm:Main-Sunflower-Free-Capacity} and \ref{thm:Main-Z_D-sunflower-free-capacity}
we bound the slice rank of a function of three variables $T(x,y,z)$ which is nonvanishing if and only if $x=y=z$ or $x,y,z$ form a sunflower. 

A function $f:A^{k}\rightarrow\mathbb{F}$, where $A^{k}=A\times A\times\cdots\times A$
denotes the cartesian product and $\mathbb{F}$ is a field, is said
to be a \emph{slice} if it can be written in the form
\[
f(x_{1},\dots,x_{k})=h(x_{i})g(x_{1},\dots,x_{i-1},x_{i+1},\dots,x_{k})
\]
where $h:A\rightarrow\mathbb{F}$ and $g:A^{k-1}\rightarrow\mathbb{F}$.
The \emph{slice rank} of a general function $f:A^{k}\rightarrow\mathbb{F}$ is
the smallest number $m$ such that $f$ is a linear combination of $m$ slices. If $A$ is a sunflower-free set, it
follows that, for $x,y,z\in A$, $T(x,y,z)$ is nonzero if and only if $x=y=z$.

We then have the following lemma:
\begin{lem}
\label{lem:Tao-rank-hyperdiagonal-matrices}(Rank of diagonal hypermatrices
\cite[Lemma 1]{TaosBlogCapsets}) Let $A$ be a finite set and $\mathbb{F}$
a field. Let $T(x,y,z)$ be a function $A \times A \times A \to \mathbb F$ such that $T(x,y,z)\neq 0$ if and only if $x=y=z$. Then the slice rank of $T$ is equal to $|A|$.
\end{lem}
Using this lemma, we need only find an upper bound on the slice rank of $T$ to obtain an upper bound on the size of the sunflower-free set. In each case we do this by an explicit decomposition of $T$ into slices found by writing $T$ as either a polynomial or as a sum of characters. We refer the reader to section 4 of \cite{BlasiakChurchCohnGrochowNaslundSawinUmans2016MatrixMultiplication} for further discussion of the slice rank.

This method is the direct analogue of Tao's interpretation \cite{TaosBlogCapsets} of the Ellenberg-Gijswijt argument for capsets, and  can be thought of as a $3$-tensor generalization of the
Haemmer bound \cite{Haemer1978AnUpperBoundForTheShannonCapacityOfAGraph},
which bounds the Sperner capacity of a hypergraph rather than the
Shannon capacity of a graph. 

We stress two differences between our result and several other papers which use the slice rank method \cite{TaosBlogCapsets} \cite{BlasiakChurchCohnGrochowNaslundSawinUmans2016MatrixMultiplication}, or which have been reinterpreted to use the slice rank \cite{CrootLevPachZ4} \cite{EllenbergGijswijtCapsets}. First, these papers study functions valued in finite fields, whose characteristic is chosen for the specific problem and cannot be changed without affecting the bound. Our work uses functions valued in a field of characteristic zero, though we could have done the same thing in any finite field of sufficiently large characteristic. Second, these papers mainly describe functions as low-degree polynomials and use that structure to bound their slice rank. In the proof of Theorem \ref{thm:Main-Z_D-sunflower-free-capacity}, we describe functions as sums of characters. One can interpret characters as polynomials restricted to the set of roots of unity, but under this interpretation the degree of the polynomial is not relevant to the proof of Theorem \ref{thm:Main-Z_D-sunflower-free-capacity} - only the number of nontrivial characters is.

\begin{rem}
The proofs of theorems \ref{thm:Main-Sunflower-Free-Capacity} and
\ref{thm:Main-Z_D-sunflower-free-capacity} can be extended without
modification to handle a multicolored version of the problem analogous to multicolored sum-free sets as defined in \cite{BlasiakChurchCohnGrochowNaslundSawinUmans2016MatrixMultiplication}\emph{. }
\end{rem}

\section{\label{sec:The-Erdos-Szemeredi-Sunflower-Problem}The Erd\H{o}s-Szemer\'{e}di
Sunflower Problem}

Any subset of $\{1,2,\dots.n\}$ corresponds to a vector in $\{0,1\}^{n}$
where a $1$ or $0$ in coordinate $i$ denotes whether or not $i$
lies in the subset. A sunflower-free collection of subsets of $\{1,2,\dots,n\}$
gives rise to a set $S\subset\{0,1\}^{n}$ with the property that
for any three distinct vectors $x,y,z\in S$, there exists $i$
such that $\{x_{i},y_{i},z_{i}\}=\{0,1,1\}$.

Moreover, a sunflower-free collection of subsets of $\{1,2,\dots,n\}$ that also does not contain two subsets with one a proper subset of the other gives rise to a set $S \subset \{0,1\}^{n}$ such that for any $x,y,z\in S$ not all equal, there exists $i$ such that $\{x_{i},y_{i},z_{i}\}=\{0,1,1\}$. This holds because the only new case is when two are equal and the third is not (say $x=y$ and $z$ is distinct), and then because $x \neq z$, $x$ is not a subset of $z$, so there exists some $i$ such that $x_i=y_i=1$ and $z_i=0$.

Given a sunflower-free
set $S\subset\{0,1\}^{n}$, let $S_{l}$, for $l=1,\dots,n$, denote
the elements of $S$ with exactly $l$ ones so that $S=\cup_{l=0}^{n}S_{l}$. Then for each $l$, $S_l$ is a sunflower-free collection of subsets with none a proper subset of another, hence whenever $x,y,z\in S_{l}$ satisfy $x+y+z\notin\{0,1,3\}^{n}$
we must have $x=y=z$. For $x,y,z\in\{0,1\}^{n}$ consider the function
$T:\left\{ 0,1\right\} ^{n}\times\left\{ 0,1\right\} ^{n}\times\left\{ 0,1\right\} ^{n}\rightarrow\mathbb{R}$
given by 
\[
T(x,y,z)=\prod_{i=1}^{n}\left(2-(x_{i}+y_{i}+z_{i})\right).
\]
The function $T(x,y,z)$ is nonvanishing precisely on triples $x,y,z$ such that there does not exist $i$ where $\{x_i,y_i,z_i\}=\{1,1,0\}$. Hence restricted to $S_l \times S_l \times S_l$, $T(x,y,z)$ is nonzero if and only if $x=y=z$. So by Lemma \ref{lem:Tao-rank-hyperdiagonal-matrices}, the slice rank of $T$ is at least
$|S_{l}|$. Expanding
 the product form for $T(x,y,z)$, we may write $T(x,y,z)$ as a linear
combination of products of three monomials 
\[
x_{1}^{i_{1}}\cdots x_{n}^{i_{n}}y_{1}^{j_{1}}\cdots y_{n}^{j_{n}}z_{1}^{k_{1}}\cdots z_{n}^{k_{n}}
\]
 where $i_{1},\dots,i_{n},j_{1},\dots,j_{n},k_{1},\dots,k_{n}\in\{0,1\}^{n}$,
and 
\[
i_{1}+\cdots+i_{n}+j_{1}+\cdots+j_{n}+k_{1}+\cdots+k_{n}\leq n.
\]
For each product of three monomials, at least one of $i_{1}+\cdots+i_{n}$, $j_{1}+\cdots+j_{n}$,
$k_{1}+\cdots+k_{n}$ is at most $n/3$. For each term in $T$, choose either $x_{1}^{i_{1}}\cdots x_{n}^{i_{n}}$, $y_{1}^{j_{1}}\cdots y_{n}^{j_{n}}$, or $z_{1}^{k_{1}}\cdots z_{n}^{k_{n}}$, making sure to choose one of total degree at most $n/3$.  Divide the expansion of $T$ into, for each possible monomial, the sum of all the terms where we chose that monomial. Because one monomial in each of these sums is fixed, we can express each sum as a product of that monomial (a function of one variable) times the sum of all the other terms (a function of the other variables), hence each of these sums is a slice. The total slice rank is at most the number of slices, which is at most the number of monomials we can choose: $3$ times the number of monomials in $n$ variables of degree at most $1$ in each variable and of total degree at most $k$. The number of such monomials is exactly $\sum_{k\leq n/3}\binom{n}{k}$, so this yields the upper bound 
\[
|S_{l}|\leq3\sum_{k\leq n/3}\binom{n}{k},
\]
\[ |S| \leq \sum_{l=0}^n |S_l| \leq 3(n+1) \sum_{k\leq n/3}\binom{n}{k}\]
which is the statement of Theorem \ref{thm:Main-Sunflower-Free-Capacity}.

\subsection{Capset Capacity}

A capset $A$ is a subset of $\mathbb{F}_{3}^{n}$ containing no three-term arithmetic progressions. Let $A_{n}\subset\mathbb{F}_{3}^{n}$
denote the largest capset in dimension in $n$, and define
\[
C=\limsup_{n\rightarrow\infty}|A_{n}|^{1/n}
\]
to be the \emph{capset capacity}. Note that $|A_n|$ is super-multiplicative, that is for $m,n\geq 1$ we have $|A_{mn}|\geq |A_{n}|^m$ since $A_n$ Cartesian-producted with itself $m$ times is a capset in $\mathbb{F}_3^{mn}$. Ellenberg and Gijswijt \cite{EllenbergGijswijtCapsets}
proved that $C\leq2.7552$, and the following theorem is a quantitative
version of a result of Alon, Sphlika, and Umans \cite[pp. 7]{AlonSphilkaUmansSunflowerMatrix}.
\begin{thm}
We have that $\mu_{3}^{S}\leq\sqrt{1+C}$ where $C$ is the capset
capacity and $\mu_{3}^{S}$ is th\emph{e }Erd\H{o}s-Szmeredi-sunflower-free
capacity.\end{thm}
\begin{proof}
We will bound the size of the largest sunflower-free set in $\{0,1\}^{2n}$
by writing each vector in terms of the four vectors in $\{0,1\}^{2}$
\[
u_{0}=\left[\begin{array}{c}
0\\
0
\end{array}\right],\ u_{1}=\left[\begin{array}{c}
1\\
0
\end{array}\right],\ u_{2}=\left[\begin{array}{c}
0\\
1
\end{array}\right],\ u_{3}=\left[\begin{array}{c}
1\\
1
\end{array}\right].
\]
Every set $S\subset\{0,1\}^{2n}$ corresponds to a set $\tilde{S}\in\{0,1,2,3\}^{n}$
where we obtain $S$ from $\tilde{S}$ by replacing each symbol $i$
for $i\in\{0,1,2,3\}$ with the vector $u_{i}$. For example, 
\[
\left[\begin{array}{c}
1\\
0\\
1\\
1
\end{array}\right]\longleftrightarrow\left[\begin{array}{c}
1\\
3
\end{array}\right]\text{ and }\left[\begin{array}{c}
0\\
1\\
0\\
0
\end{array}\right]\longleftrightarrow\left[\begin{array}{c}
2\\
0
\end{array}\right].
\]
 For each $x\in\{0,1\}^{n}$ consider 
\[
\tilde{S}_{x}=\left\{ v\in\tilde{S}:\ v_{i}=3\text{ if and only if }x_{i}=1\right\} .
\]
We may view elements of $\tilde{S}_x$ as elements of $\{0,1,2\}^{n-x} =\mathbb F_3^{n-x}$ by ignoring the coordinates where $x$ is $1$. If three elements in $\tilde{S}_x$ form an arithmetic progression in $\mathbb F_3^{n-x}$, then in each coordinate the elements of $\tilde{S}_x$ are either all the same or are $0,1,2$ in any order, so the entries of the corresponding vectors in $S$ are either all the same or $u_0,u_1,u_2$ in any order. Because $u_{0},u_{1},u_{2}$ form a sunflower, these three elements of $S$ form a sunflower. Because $S$ is a sunflower-free set, $\tilde{S}_x$ is a capset. Let $w(x)=\sum_{i=1}^{n}x_{i}$ be the weight of the vector $x$, then
\[
|\tilde{S}_{x}|\leq C^{n-w(x)}
\]
where $C$ is the capset capacity. Hence 
\[
|S|\leq\sum_{x}C^{n-w(x)}=\sum_{j=0}^{n}\binom{n}{j}C^{n-j}=(1+C)^n,
\]
and we obtain the desired bound.
\end{proof}

Using the Ellenberg-Gijswijt upper bound on capset capacity, this gives $\mu_3^S \leq 1.938$, which is not as strong a bound as Theorem \ref{thm:Main-Sunflower-Free-Capacity}.

\section{\label{sec:Sunflower-Free-Sets-in-Z_D}Sunflower-Free Sets in $(\mathbb Z/D \mathbb Z)^{n}$}

Consider the $D$ characters $\chi:\mathbb{Z}/D\mathbb{Z}\rightarrow\mathbb{C}^{\times}$.
By the orthogonality relations, for any $a,b\in\mathbb{Z}/D\mathbb{Z}$
\[
\frac{1}{|D|}\sum_{\chi}\chi(a-b)=\begin{cases}
1 & \text{if}\ a=b\\
0 & \text{otherwise}
\end{cases}.
\]
Hence 
\[
\frac{1}{|D|}\sum_{\chi}\left(\chi(a)\overline{\chi(b)}+\chi(b)\overline{\chi(c)}+\chi(a)\overline{\chi(c)}\right)=\begin{cases}
0 & \text{if }a,b,c\text{ are distinct}\\
1 & \text{if exactly two of }a,b,c\text{ are equal}\\
3 & \text{if }a=b=c
\end{cases}.
\]
For $x,y,z\in(\mathbb Z/D \mathbb Z)^{n}$, define the function $T:(\mathbb Z/D \mathbb Z)^{n}\times(\mathbb Z/D \mathbb Z)^{n}\times(\mathbb Z/D \mathbb Z)^{n}\rightarrow\mathbb{C}$
by
\begin{equation}
T(x,y,z)=\prod_{j=1}^{n}\left(\frac{1}{|D|}\sum_{\chi}\left(\chi(a)\overline{\chi(b)}+\chi(b)\overline{\chi(c)}+\chi(a)\overline{\chi(c)}\right)-1\right) ,\label{eq:product_for_T}
\end{equation}
\[= \prod_{j=1}^{n}\left(\frac{1}{|D|}\sum_{\chi}\left(\chi(a)\overline{\chi(b)}1(c)+1(a)\chi(b)\overline{\chi(c)}+\chi(a)1(b)\overline{\chi(c)}\right)-1(a)1(b)1(c)\right)\]

which is non-zero if and only if $x,y,z$ form a $\mathbb Z/D \mathbb Z$-sunflower or all equal.
Let $A\subset(\mathbb Z/D \mathbb Z)^{n}$ be a sunflower free set. Then restricted to $A \times A \times A$, $T$ is nonzero if and only if  $x=y=z$. Hence by
Lemma \ref{lem:Tao-rank-hyperdiagonal-matrices} the slice rank of
$T$ is at least $|A|$. Expanding the product in (\ref{eq:product_for_T}),
we see that $T$ can be written as a linear combination of terms of
the form
\[
\chi_{1}(x_{1})\cdots\chi_{n}(x_{n})\psi_{1}(y_{1})\cdots\psi_{n}(y_{n})\xi_{1}(z_{1})\cdots\xi_{n}(z_{n})
\]
where $\chi_{1},\dots,\chi_{n},\psi_{1},\dots,\psi_{n},\xi_{1},\dots,\xi_{n}$
are characters on $\mathbb{Z}/D\mathbb{Z}$, at most $2n$ of which are nontrivial.
For any such term, at least one of the tuples $\chi_1,\dots,\chi_n$, $\psi_1,\dots,\psi_n$, $\xi_1,\dots,\xi_n$ must contain at most $2n/3$ nontrivial characters. Grouping the terms by the one containing the fewest nontrivial characters like this, we are able
to upper bound the slice rank of $T$ by 
\[
3\sum_{k\leq2n/3}\binom{n}{k}(D-1)^{k},
\]
where the $(D-1)^{k}$ comes from the fact that for each set of $k$ indices, we have $(D-1)^{k}$ possible choices of non-trivial
characters. Because $D \geq 3$, $\frac{D-1}{2} \geq 1$, so:
\[
\sum_{k\leq2n/3}\binom{n}{k}(D-1)^{k} \leq \sum_{k\leq2n/3}\binom{n}{k}(D-1)^{k} \left(\frac{D-1}{2}\right)^{2n/3-k}=   \left(\frac{D-1}{2} \right)^{-\frac{n}{3}} \sum_{k\leq2n/3}\binom{n}{k}(D-1)^{k} \left( \frac{D-1}{2} \right)^{n-k}
\]\[
\leq \left(\frac{D-1}{2} \right)^{-\frac{n}{3}} \sum_{k\leq n}\binom{n}{k}(D-1)^{k} \left( \frac{D-1}{2} \right)^{n-k} = \left(\frac{D-1}{2} \right)^{-\frac{n}{3}}  \left( D-1 + \frac{D-1}{2} \right)^n = \left(\frac{3}{2^{2/3}}(D-1)^{2/3}\right)^{n}
\]
Let $c_{D}=\frac{3}{2^{2/3}}(D-1)^{2/3}.$ This inequality proves that for $|A|$ a sunflower-free set,
\[
|A|\leq 3c_{D}^{n}
\]
 To prove Theorem \ref{thm:Main-Z_D-sunflower-free-capacity} (that $|A| \leq c_D^n$), we can remove the factor of $3$ by a standard amplification argument, as for $A$ a sunflower-free set in $(\mathbb Z/D \mathbb Z)^n$, $A^k$ is a sunflower-free set in $(\mathbb Z/D \mathbb Z)^{nk}$, so $|A| \leq (3c_D^{nk})^{1/k} = 3^{1/k} c_D^n$. Taking $k\rightarrow\infty$, we obtain Theorem \ref{thm:Main-Z_D-sunflower-free-capacity}.

\bibliographystyle{plain}

\begin{thebibliography}{1}

\bibitem{AlonSphilkaUmansSunflowerMatrix}
Noga Alon, Amir Shpilka, and Christopher Umans.
\newblock On sunflowers and matrix multiplication.
\newblock {\em Comput. Complexity}, 22(2):219--243, 2013.

\bibitem{BlasiakChurchCohnGrochowNaslundSawinUmans2016MatrixMultiplication}
Jonah Blasiak, Thomas Church, Henry Cohn, Joshua Grochow, Eric Naslund, Will
  Sawin, and Christopher Umans.
\newblock On cap sets and the group-theoretic approach to matrix
  multiplication.
\newblock 2016.
\newblock URL:https://arxiv.org/abs/1605.06702/.

\bibitem{CrootLevPachZ4}
Ernie Croot, Vsevolod Lev, and Peter Pach.
\newblock Progression-free sets in $\mathbb Z_4^n$ are exponentially small.
\newblock 2016.
\newblock URL:https://arxiv.org/abs/1605.01506/.

\bibitem{EllenbergGijswijtCapsets}
Jordan Ellenberg and Dion Gijswijt.
\newblock On large subsets of $\mathbb F_q^n$ with no three-term arithmetic
  progression.
\newblock 2016.
\newblock URL:https://arxiv.org/abs/1605.09223/.

\bibitem{ErdosRadoTheorem}
P.~Erd{\H{o}}s and R.~Rado.
\newblock Intersection theorems for systems of sets.
\newblock {\em J. London Math. Soc.}, 35:85--90, 1960.

\bibitem{ErdosSzemeredi1978UpperBoundsForSunflowerFreeSets}
P.~Erd{\H{o}}s and E.~Szemer{\'e}di.
\newblock Combinatorial properties of systems of sets.
\newblock {\em J. Combinatorial Theory Ser. A}, 24(3):308--313, 1978.

\bibitem{Haemer1978AnUpperBoundForTheShannonCapacityOfAGraph}
W.~Haemers.
\newblock An upper bound for the {S}hannon capacity of a graph.
\newblock In {\em Algebraic methods in graph theory, {V}ol. {I}, {II}
  ({S}zeged, 1978)}, volume~25 of {\em Colloq. Math. Soc. J\'anos Bolyai},
  pages 267--272. North-Holland, Amsterdam-New York, 1981.

\bibitem{NaslundLowerBounds}
Eric Naslund.
\newblock Lower bounds for capsets and sunflower-free sets.
\newblock Unpublished.

\bibitem{TaosBlogCapsets}
Terence Tao.
\newblock A symmetric formulation of the croot-lev-pach-ellenberg-gijswijt
  capset bound, 2016.
\newblock
  URL:https://terrytao.wordpress.com/2016/05/18/a-symmetric-formulation-of-the-croot-lev-pach-ellenberg-gijswijt-capset-bound/.

\end{thebibliography}

\end{document}